\documentclass[11pt,oneside]{article}

\usepackage{amsthm, amsmath, amssymb, amsfonts,epsfig,mathrsfs}
\usepackage{epstopdf}
\usepackage{float}
\restylefloat{table}

\usepackage{graphicx}
\usepackage[font=sl,labelfont=bf]{caption}
\usepackage{subcaption}
\usepackage{slashbox}

\usepackage{enumerate}
\usepackage[colorlinks=true, pdfstartview=FitV, linkcolor=blue,
            citecolor=blue, urlcolor=blue]{hyperref}
\usepackage[usenames]{color}
\definecolor{Red}{rgb}{0.7,0,0.1}
\definecolor{Green}{rgb}{0,0.7,0}
\usepackage{url}
\makeatletter\def\url@leostyle{%
 \@ifundefined{selectfont}{\def\UrlFont{\sf}}{\def\UrlFont{\scriptsize\ttfamily}}} \makeatother
\usepackage{accents, comment}
\usepackage{bbm}
\usepackage{mathrsfs}
\usepackage[margin=1.1in, dvips, letterpaper]{geometry}

\newtheorem{theorem}{Theorem}

\theoremstyle{definition}

\theoremstyle{remark}

\numberwithin{equation}{section}
\numberwithin{theorem}{section}
\def\cB{\mathcal{B}}

\def\cF{\mathcal{F}}

\def\cH{\mathcal{H}}

\def\cK{\mathcal{K}}
\def\cL{\mathcal{L}}

\def\cP{\mathcal{P}}

\def\cR{\mathcal{R}}

\def\bE{\mathbb{E}}

\def\bN{\mathbb{N}}

\def\bP{\mathbb{P}}
\def\bQ{\mathbb{Q}}
\def\bR{\mathbb{R}}

\newcommand{\1}{\mathbbm{1}}                     % preferable way of writing indicator function
\newcommand{\set}[1]{\{#1\}}            % set: {xyz} to be used for inline formulas
\renewcommand{\d}{\operatorname{d}\!}   % d in dt
\title{A note on error estimation for hypothesis testing problems \\for some linear SPDEs}
\author{
Igor Cialenco \\
\small  Department of Applied Mathematics \\[-0.6ex]
\small  Illinois Institute of Technology \\[-0.6ex]
\small  10 West 32nd Str, Bld E1, Room 208 \\[-0.6ex]
\small  Chicago, IL 60616-3793  \\[-0.6ex]
\small  \url{igor@math.iit.edu}
\and %\mbox{}\\  %% use either \mbox{}\\ or \and to have one col or two cols with the author
 Liaosha Xu\\
\small  Department of Applied Mathematics \\[-0.6ex]
\small  Illinois Institute of Technology \\[-0.6ex]
\small  10 West 32nd Str, Bld E1, Room 208 \\[-0.6ex]
\small  Chicago, IL 60616-3793  \\[-0.6ex]
\small  \url{lxu29@hawk.iit.edu}
}

\date{First Circulated: December 15, 2013\\ This version: July 15, 2014\\
Forthcoming in Stochastic Partial Differential Equations: Analysis and Computations}
\begin{document}
\maketitle

\begin{abstract}
\noindent
The aim of the present paper is to estimate and control the Type I and Type II errors of a simple hypothesis testing problem of the drift/viscosity coefficient for stochastic fractional heat equation driven by additive noise. Assuming that one path of the first $N$ Fourier modes of the solution is observed continuously over a finite time interval $[0,T]$, we propose a new class of rejection regions and provide computable thresholds for $T$, and $N$, that guarantee that the statistical errors are smaller than a given upper bound. The considered tests are of likelihood ratio type. The main ideas, and the proofs, are based on sharp large deviation bounds. Finally, we illustrate the theoretical results by numerical simulations.

%
%
%for finite $T$ and $N$, in a feasible and practical way.
%
%\textcolor[rgb]{0.00,0.07,1.00}{In the recent paper \cite{CialencoXu2013a}, we study the simple hypothesis testing problem for the drift/viscosity coefficient for stochastic fractional heat equation driven by additive space-time white noise colored in space. Assuming that one path of the projected solution is observed continuously over time interval $[0,T]$, we established `the proper asymptotic classes' of tests in which we can find `asymptotically the most powerful test' -- tests with fastest speed of error convergence -- in two asymptotic regimes: large time asymptotics $T\to\infty$, and increasing number of Fourier modes $N\to\infty$. The main goal of the present paper is to study how to estimate and control the Type I and Type II errors for finite $T$ and $N$, in a feasible and practical way.
%We propose a new Likelihood Ratio type rejection regions and derive explicit expressions for thresholds for $T$, and $N$ that will guarantee that the corresponding statistical errors are smaller than a given upper bound.
%The key ideas of the proofs are essentially based on some results on sharp large deviation principles developed in \cite{CialencoXu2013a}.  The theoretical results are illustrated by means of numerical simulations.}

\bigskip
{\noindent \small
{\it \bf Keywords:} Hypothesis testing for SPDE; Likelihood Ratio; Maximum Likelihood Estimator; error estimates; sharp large deviation; fractional heat equation; additive space-time white noise.

\smallskip
\noindent {\it \bf MSC2010:} 60H15, 35Q30, 65L09}

%\newpage

\end{abstract}
%\tableofcontents

\section{Introduction}

%The statistical inference problems for Stochastic Partial Differential Equation (SPDEs) have been studied by many authors

Under assumption that one path of the first $N$ Fourier modes of the solution of a Stochastic Partial Differential Equation (SPDE) is observed continuously over a finite time interval, the parameter estimation problem for the drift coefficient  has been studied by several authors, starting with the seminal paper \cite{HubnerRozovskiiKhasminskii}. Consistency and asymptotic normality of the MLE type estimators are well understood, at least for equations driven by additive noise; see for instance the survey paper \cite{Lototsky2009Survey} for linear SPDEs, and \cite{IgorNathanAditiveNS2010} for nonlinear equations, and references therein. Generally speaking, the statistical inference theory for SPDEs did not go far beyond the fundamental properties of MLE estimators, although important and interesting classes of SPDEs driven by various noises were studied.
The first attempt to study hypothesis testing problem for SPDEs is due to \cite{CialencoXu2013a}, where we investigated the simple hypothesis for the drift/viscosity coefficient for stochastic fractional heat equation driven by additive noise, white in time and colored in space.
Therein, the authors established `the proper asymptotic classes' of tests for which we can find `asymptotically the most powerful tests' -- tests with fastest speed of error convergence.
Moreover, we provided explicit forms of such tests in two asymptotic regimes: large time asymptotics $T\to\infty$, and increasing number of Fourier modes $N\to\infty$.
By its very nature, the theory developed in \cite{CialencoXu2013a} is based on asymptotic behavior, $T,N\to\infty$, and a follow-up question is how large $T$ or $N$ should we take, such that the Type~I and Type~II errors of these tests are smaller than a given threshold.
The main goal of this paper is to develop feasible methods to estimate and control the Type I and Type II errors when $T$ and $N$ are finite.
Similar to \cite{CialencoXu2013a}, we are interested in Likelihood Ratio type rejection regions $R_{T}=\{U_T^N: \ln L(\theta_0,\theta_1,U_T^N)\ge\eta T\}$ and $R_{N}=\{U_T^N: \ln L(\theta_0,\theta_1,U_T^N)\ge \zeta M_N\}$, where $U_T^N$ is the projected solution on the space generated by the first $N$ Fourier modes, $L$ is the Likelihood Ratio, $M_N$ is a constant that depends on the first $N$ eigenvalues of the Laplacian, and $\eta,\zeta$ are some constants that depend on $T$ and $N$.
We will derive explicit expression for $\eta$ and $\zeta$, and thresholds for $T$, and respectively for $N$, that will guarantee that the corresponding statistical errors are smaller than a given upper bound. However, this comes at the cost that these tests are no longer the most powerful in the class of tests proposed in \cite{CialencoXu2013a}.
The key ideas, and the proofs of main results, are based on  sharp large deviation principles (both in time and spectral spatial component) developed in \cite{CialencoXu2013a}.
On top of the theoretical part, we also present some numerical experiments as a coarse verification of the main theorems. We find some bounds for the numerical approximation errors, that  will also serve as a preliminary effort in studying the statistical inferences problems for SPDEs under discrete observations.
Finally, we want to mention that the case of large $T$ and $N=1$ corresponds to classical one dimensional  Ornstein-Uhlenbeck process, and even in this case, to our best knowledge, the obtained results are novel.

%\underline{The results obtain in the present paper are primarily designed for practical purposes, and we illustrate them by means of numerical simulations.}
%(The results obtained in the present paper are primarily designed for controlling the errors by certain bounds with finite $T$ and $N$.)

%In the last section we briefly comment on some possible solutions to this subject for future study.

The paper is organized as follows.  In Section~\ref{sec:1-1} we set up the problem, introduce some necessary notations, and discuss why for the tests proposed in \cite{CialencoXu2013a} it is hard to find explicit expressions for $T$ and $N$ in order to control the statistical errors.
Since sharp large deviation principles from \cite{CialencoXu2013a}  play fundamental role in the derivation of main results, in Section~\ref{sec:1-2} we briefly present them here too.  Section~\ref{sec:largeT} is devoted to the case of large time asymptotics, with number of observable Fourier modes $N$ being fixed. We show how to choose $T$ and $\eta$ such that both Type~I and Type~II errors, associated with rejection region $R_T$, are bounded by a given threshold. Similarly, in Section~\ref{sec:largeN} we study the case of large $N$ while keeping the time horizon $T$ fixed.
In Section~\ref{sec:NumericalSim} we illustrate the theoretical results by means of numerical simulations.
We start,  with the description of the numerical methods, and derive some error bounds of the numerical approximations.
Consequently, we show that while the thresholds for $T,N$ derived in Sections~\ref{sec:largeT} and \ref{sec:largeN} are conservative, as one may expect, they still provide a robust practical framework for controlling the statistical errors.
Finally, in Section~\ref{sec:Conclusion} we discuss the advantages and drawbacks of the current results and briefly elaborate on possible theoretical and practical methods of solving some of the open problems.

%While the sharp deviations results for large time asymptotics $T\to\infty$ are comparable in certain respects with those from Stochastic ODEs (cf. \cite{BercuRouault2001,KutoyantsBook2004,Linckov1999}), the results for large number of Fourier modes $N\to\infty$ are new, and by analogy we refer to them also as sharp large deviation principle.

\subsection{Setup of the problem and some auxiliary results}\label{sec:1-1}
In this section we will set up the main equation, briefly recall the problem settings of hypothesis testing for the drift coefficient, and present some needed results from  \cite{CialencoXu2013a}. Also here we give the motivations that lead to the proposed problems.

Similar to \cite{CialencoXu2013a}, on a filtered probability space $(\Omega,\cF,\set{\cF_t}_{t\geq0},\bP)$ we considered the following stochastic evolution equation
\begin{equation}\label{eq:mainSPDE}
\d U(t,x) + \theta (-\Delta)^\beta U(t,x)\d t = \sigma \sum_{k\in\bN} \lambda_k^{-\gamma}h_k(x)\d w_k(t), \quad t\in[0,T], \ U(0,x) = U_0, \ x\in G,
\end{equation}
where $\theta>0$, $\beta>0, \ \gamma \geq 0$, $\sigma\in\bR\setminus\set{0}$, $U_0\in H^s(G)$ for some $s\in\bR$,
$w_j$'s are independent standard Brownian motions, $G$ is a bounded and smooth domain in $\bR^d$, $\Delta$ is the Laplace operator on $G$ with zero boundary conditions,
$h_k$'s are eigenfunctions of $\Delta$. It is well known that $\set{h_k}_{k\in\bN}$ form a complete orthonormal system in $L^2(G)$.
We denote by $\rho_k$ the eigenvalue corresponding to $h_k$, and put $\lambda_k:=\sqrt{-\rho_k}, \ k\in\bN$.
Under some fairly general assumptions, equation \eqref{eq:mainSPDE} admits a unique solution in the appropriate Sobolev spaces (see for instance \cite{CialencoXu2013a}).

We assume that all parameters are known, except the drift/viscosity coefficient $\theta$ which is the parameter of interest, and we use the spectral approach (for more details see the survey paper \cite{Lototsky2009Survey}) to derive MLE type estimators for $\theta$.
In what follows, we denote by $u_k,k\in\bN,$ the Fourier coefficient of the solution $u$ of \eqref{eq:mainSPDE} with respect to $h_k,k\in\bN$, i.e. $u_k(t) = (U(t),h_k)_0, k\in\bN$.
Let $H_N$ be the finite dimensional subspace of $L_2(G)$ generated by $\set{h_k}_{k=1}^N$, and denote by $P_N$ the projection operator of $L_2(G)$ into $H_N$, and put $U^N = P_NU$, or equivalently $U^N:=(u_1,\ldots,u_N)$. Note that each Fourier mode $u_k,k\in\bN$, is an Ornstein–Uhlenbeck process with dynamics given by
\begin{equation}\label{eq:OU-Fourier}
\d u_k = -\theta \lambda_k^{2\beta} u_k\d t + \sigma \lambda_k^{-\gamma} \d w_k(t), \quad  u_k(0) = (U_0,h_k), \ t\geq 0.
\end{equation}
We denote by $\bP^{N,T}_{\theta}$ the probability measure on $C([0,T]; H_N)\backsimeq C([0,T]; \bR^N)$ generated by $U^N$.
The measures $\bP^{N,T}_{\theta}, \ \theta>0$, are equivalent to each other, with the Radon-Nikodym derivative, or the Likelihood Ratio, of the form
\begin{align}\label{eq:RadonNikodymUn}
L(\theta_0,\theta;U^N_T)=\frac{\bP^{N,T}_{\theta}}{\bP^{N,T}_{\theta_{0}}} =
\exp\Big(&-(\theta-\theta_0)\sigma^{-2}\sum_{k=1}^N\lambda_k^{2\beta+2\gamma}\big(\int_0^Tu_k(t)du_k(t)  \nonumber \\
&\qquad +\frac{1}{2}(\theta+\theta_0)\lambda_k^{2\beta}\int_0^Tu_k^2(t)dt\big)\Big),
\end{align}
where $U_T^N$ denotes the trajectory of $U_N$ over the time interval $[0,T]$.
Maximizing the Log of the Likelihood Ratio with respect to $\theta$, we get the Maximum Likelihood Estimator (MLE)
\begin{equation}\label{eq:MLE-UN}
\hat{\theta}_T^N = -\frac{\sum_{k=1}^{N}\lambda_k^{2\beta+2\gamma}\int_0^T u_k(t)du_k(t)}{\sum_{k=1}^{N}\lambda_k^{4\beta+2\gamma}\int_0^T u_k^2(t)dt},
\quad N\in\bN, T>0.
\end{equation}
In \cite{CialencoXu2013a}, we established the strong consistency and asymptotic normality of the MLE, when $T$ or $N$ goes to infinity.

In this work we consider a simple hypothesis testing problem for $\theta$, assuming that the parameter $\theta$ can take only two values $\theta_0,\theta_1$,
with  the null and the alternative hypothesis as follows
\begin{align*}
\mathscr{H}_0 & :\quad \theta=\theta_0, \\
\mathscr{H}_1 &:\quad \theta=\theta_1.
\end{align*}
Without loss of generality,  we will assume that $\theta_1>\theta_0$, and $\sigma>0$.
Throughout, we fix a significant level $\alpha\in(0,1)$.
Suppose that $R\in\cB(C([0,T];\bR^N))$ is a rejection region for the test, i.e. if  $U_T^N\in R$ we reject the null and accept the alternative.
The quantity $\bP_{\theta_0}^{N,T}(R)$ is called the Type~I error of the test $R$, and respectively $1-\bP_{\theta_1}^{N,T}(R)$ is called the Type~II error.
Naturally, we seek rejection regions with Type~I error  smaller than the significance level $\alpha$:
$$
\mathcal{K}_{\alpha}:=\left\{R\in\cB(C([0,T];\bR^N)): \bP^{N,T}_{\theta_0}(R)\leq\alpha\right\}.
$$
Let us denote by $R^*$ the rejection region (likelihood ratio test) of the form
$$
R^*=\{U_T^N: L(\theta_0,\theta_1,U_T^N)\ge c_{\alpha}\},
$$
where $c_\alpha\in\bR$, such that $\bP^{N,T}_{\theta_0}(L(\theta_0,\theta_1,U_T^N)\ge c_{\alpha})=\alpha$.
In \cite{CialencoXu2013a} we proved that $R^*$ is the most powerful test (has the smallest Type~II error) in the class $\cK_\alpha$,
$$
\bP^{N,T}_{\theta_1}(R)\le\bP^{N,T}_{\theta_1}(R^*),\qquad \textrm{ for all } R\in\mathcal{K}_{\alpha}.
$$
While this gives a complete theoretical answer to the simple hypothesis testing problem, generally speaking there is no explicit formula for the constant $c_\alpha$.
The main contribution of \cite{CialencoXu2013a} was to find computable rejection regions, and the appropriate class of tests, by so called asymptotic approach.
The authors study two asymptotic regimes: large time asymptotics, while fixing the number of Fourier modes $N$, and large number of Fourier modes, while time horizon is fixed.
We will outline here the case of large time asymptotics.
Let $(R_T^\sharp)_{T\in\bR_+}$ and $\mathcal{K}_{\alpha}^\sharp$ be defined as follows:
\begin{align*}
\mathcal{K}_{\alpha}^\sharp & =\left\{(R_T): \limsup_{T\to\infty}\left(\bP^{N,T}_{\theta_0}(R_T)-\alpha\right)\sqrt{T}\le\alpha_1\right\},  \\
R_T^\sharp & =\left\{U_T^N: L(\theta_0,\theta_1,U_T^N)\ge c^\sharp_{\alpha}(T)\right\},\\
c^\sharp_{\alpha}(T) &=\exp\left(-\frac{(\theta_1-\theta_0)^2}{4\theta_0}MT-\frac{\theta_1^2-\theta_0^2} {2\theta_0}\sqrt{\frac{MT}{2\theta_0}}q_{\alpha}\right), \\
M & =\sum_{k=1}^N\lambda_k^{2\beta},
\end{align*}
where $q_\alpha$ is $\alpha$-quantile of standard Gaussian distribution, and $\alpha_1$ is a constant that depends on $\alpha$.
The class $\cK_\alpha^\sharp$ essentially consists of tests with Type~I errors converging to $\alpha$ from above with rate at least $\alpha_1T^{-1/2}$.
It was proved that
\begin{align}\label{eq:asym-def}
\underset{T\to\infty}{\liminf} \frac{1-\bP^{N,T}_{\theta_1}(R_T)}{1-\bP^{N,T}_{\theta_1}(R_{T}^\sharp)}\ge1,\qquad \textrm{ for all } (R_T)_{T\in\bR_+}\in\mathcal{K}_{\alpha}^\sharp.
\end{align}
In other words, $R_T^\sharp$ has the fastest rate of convergence of the Type~II error, as $T\to\infty$, in the class $\mathcal{K}_{\alpha}^\sharp$.
We proved analogous  results for $N\to\infty$, and $T$ being fixed, by taking
\begin{align*}
{R}_N^\sharp & =\left\{U_T^N: L(\theta_0,\theta_1,U_T^N)\ge \widetilde{c}_{\alpha}(N)\right\}, \quad N\in\bN, \\
\widetilde{\cK}_\alpha^\sharp & = \left\{(R_N): \limsup_{N\to\infty}\left(\bP^{N,T}_{\theta_0}(R_N)-\alpha\right)\sqrt{M}\le\widetilde{\alpha}_1\right\},
\end{align*}
where $\widetilde{c}_{\alpha}(N)$ is a constant depending on $N$ and $\alpha$ only, and $\widetilde{\alpha}_1$ is a constant that depends on $\alpha$.
We refer the reader to \cite{CialencoXu2013a} for further details.

However, by their very  nature of being asymptotic type results,  one cannot assess how large $T$ (or $N$) shall be taken to guarantee that the error is smaller than a desired tolerance. \textit{The main goal of this manuscript is to investigate the corresponding error estimates for fixed values of $T$ and $N$}.

Let us start with some heuristic discussion on why for the tests $R^\sharp_T$ and $R^\sharp_N$ one cannot easily find computable  expressions for $T$ or $N$ that will guarantee  certain bounds on statistical errors.
As it was shown in \cite[Lemma 3.13]{CialencoXu2013a}, for sufficiently large $T$, we have the following asymptotic expansion under the null hypothesis $\mathscr{H}_0$:
\begin{align*}
\bP^{N,T}_{\theta_0}(R_T^\sharp)=\alpha+\alpha_1T^{-1/2}+O(T^{-1}).
\end{align*}
Hence, for $T$ large enough, we will have the estimate
\begin{align*}
\left|\bP^{N,T}_{\theta_0}(R_T^\sharp)-\alpha\right|\le C_1T^{-1/2},
\end{align*}
where $C_1$ is a constant independent of $T$.
Similarly (cf. \cite[Lemma 3.21]{CialencoXu2013a}), we have the asymptotic expansions
\begin{align*}
\bP^{N,T}_{\theta_0}({R}_N^\sharp)=&\alpha+\widetilde{\alpha}_1M^{-1/2}+o(M^{-1/2}), \quad \textrm{ if } \ \beta/d>1/2, \\
\bP^{N,T}_{\theta_0}({R}_N^\sharp)=&\alpha+ \left(\widetilde{\alpha}_1+\sqrt{\frac{2\beta/d+1}{c^{\beta}}}\widetilde{\alpha}_2\right)M^{-1/2}+o(M^{-1/2}), \quad \textrm{ if } \ \beta/d=1/2.
\end{align*}
Since $\lambda_k\sim k^{1/d}$, for $\beta/d\ge1/2$, we get
\begin{align*}
\left|\bP^{N,T}_{\theta_0}( {R}_N^\sharp)-\alpha\right|\le C_2N^{-\beta/d-1/2},
\end{align*}
where $C_2$ is a constant independent of $N$.

Due to lack of knowledge  of the behavior of higher order terms in the above asymptotics, practically speaking, the above constants $C_1$ and $C_2$ cannot be easily determined.
The case of large Fourier modes is especially intricate, since the asymptotic expansion of Type~I error is done in terms of $M$ rather than $N$.
\textit{To overcome this technical problem, we propose a new test, which may not be asymptotically the most powerful, but which is convenient for the errors' estimation.
Moreover, we validate the obtained results by numerical simulations.}

\subsection{Sharp Large Deviation Principle}\label{sec:1-2}

The main results presented in this paper, and the ideas behind them, rely on some results on sharp large deviation bounds obtained in \cite{CialencoXu2013a}.
While the sharp deviations results for large time asymptotics $T\to\infty$ are comparable in certain respects with those from Stochastic ODEs (cf. \cite{BercuRouault2001,KutoyantsBook2004,Linckov1999}), the results for large number of Fourier modes $N\to\infty$ are new, and by analogy we refer to them also as sharp large deviation principle.
For convenience, we will briefly present some of needed results here too.

Generally speaking, we seek asymptotics expansion of the form
\begin{align*}
T^{-1}\ln\bE_{\theta}\left[\exp\left(\epsilon\ln L(\theta_0,\theta_1,U_T^N)\right)\right]
=\mathcal{L}(\epsilon)+T^{-1}\mathcal{H}(\epsilon)+T^{-1}\mathcal{R}(\epsilon),
\end{align*}
for $\theta=\theta_0$ or $\theta = \theta_1$, and where $\cL, \ \cH$ are some explicit function of $\epsilon, N, \theta_0,\theta_1$, and $\cR$ is a residual term.
Similarly, we are looking for asymptotic expansion of $M^{-1}\ln\bE_{\theta_1}\left[\exp\left(\epsilon\ln L(\theta_0,\theta_1,U_T^N)\right)\right]$, while $T$ is fixed.
With these at hand, we find a convenient representation of probabilities
$$
\bP^{N,T}_{\theta_j}\left(\ln L(\theta_0,\theta_1,U_T^N) \leq (\textrm{or}\ge) \varpi\right), \quad j=0,1,
$$
where $\varpi$ has the form $\eta T$ or $\eta M$ for some constant $\eta$.
%We refer the reader to \cite{CialencoXu2013a} for details on derivation of these technical results.
Below we will present the explicit expressions for functions $\cL,\cH,\cR$.
Albeit the formulas are somehow cumbersome, their particular form is less important at this stage.

Along these lines, we adapt the  notations
\begin{align*}
  %G(U_T^N):= & \ln L(\theta_0,\theta_1,U_T^N),  \\
  \mathcal{L}_T^j(\epsilon):=&T^{-1}\ln\bE_{\theta_j}\left[\exp\left(\epsilon\ln L(\theta_0,\theta_1,U_T^N)\right)\right],  \\
  \mathcal{L}_N^j(\epsilon):=&M^{-1}\ln\bE_{\theta_j}\left[\exp\left(\epsilon\ln L(\theta_0,\theta_1,U_T^N)\right)\right],
\end{align*}
for $j=0,1$.
The following expansions hold true
\begin{align}\label{eq:MomentLHRSplit}
\mathcal{L}_T^j(\epsilon):=&M\mathcal{L}_j(\epsilon)+T^{-1}N\mathcal{H}_j(\epsilon)+T^{-1}\mathcal{R}_j(\epsilon), \\
\mathcal{L}_N^j(\epsilon):=&T\mathcal{L}_j(\epsilon)+NM^{-1}\mathcal{H}_j(\epsilon)+M^{-1}\mathcal{R}_j(\epsilon),
\end{align}
where $\epsilon>-\frac{\theta_j^2}{\theta_1^2-\theta_0^2}$, and where
\begin{align*}
\mathcal{L}_j(\epsilon)=&\frac{1}{2}\left(\theta_j+(\theta_1-\theta_0)\epsilon-
\sqrt{\theta_j^2+(\theta_1^2-\theta_0^2)\epsilon}\right),
\\
\mathcal{H}_j(\epsilon)=&
-\frac{1}{2}\ln\left(\frac{1}{2}+\frac{1}{2}\mathcal{D}_j(\epsilon)\right),\qquad \mathcal{D}_j(\epsilon)=\frac{\theta_j+(\theta_1-\theta_0)\epsilon}{\sqrt{\theta_j^2+(\theta_1^2-\theta_0^2)\epsilon}},
\\
\mathcal{R}_j(\epsilon)=&-\frac{1}{2}\sum_{k=1}^N\ln\left(1+\frac{1-\mathcal{D}_j(\epsilon)}{1+\mathcal{D}_j(\epsilon)} \exp\left(-2\lambda_k^{2\beta}T\sqrt{\theta_j^2+(\theta_1^2-\theta_0^2)\epsilon}\right)\right).
\end{align*}
Using these results, one can show that the following identities are satisfied,
\begin{align}
\bP^{N,T}_{\theta_j}&\left((-1)^j\ln L(\theta_0,\theta_1,U_T^N)\ge (-1)^j\eta T\right)=A_T^jB_T^j,\label{eq:ABSplit-RT}
\\
\bP^{N,T}_{\theta_j}& \left((-1)^j\ln L(\theta_0,\theta_1,U_T^N)\ge (-1)^j\eta M\right)=\widetilde{A}_N^j\widetilde{B}_N^j,\label{eq:ABSplit-RN}
\end{align}
with
\begin{align}\label{eq:ATBT-Form}
A_T^j & =\exp\left[T(\mathcal{L}_T^j(\epsilon_{\eta}^j)-\eta\epsilon_{\eta}^j)\right],\qquad \widetilde{A}_N^j= \exp\left[M(\mathcal{L}_N^j(\widetilde{\epsilon}_{\eta}^j) -\eta\widetilde{\epsilon}_{\eta}^j) \right], \notag
\\
B_T^j & =\bE_T^j\left(\exp\left[-\epsilon_{\eta}^j(\ln L(\theta_0,\theta_1,U_T^N)-\eta T)\right]\1_{\{(-1)^j\ln L(\theta_0,\theta_1,U_T^N)\ge (-1)^j\eta T\}}\right),\notag
\\
\widetilde{B}_N^j &=\bE_N^j\left(\exp\left[-\widetilde{\epsilon}_{\eta}^j(\ln L(\theta_0,\theta_1,U_T^N)-\eta M)\right]\1_{\{(-1)^j\ln L(\theta_0,\theta_1,U_T^N)\ge (-1)^j\eta M\}}\right),
\end{align}
where $\eta$ is a number which may depend on $T$ and $N$, $\epsilon_{\eta}^j$ and $\widetilde{\epsilon}_{\eta}^j$ are numbers which depend on $\eta$, $\bE_T^j$ and $\bE_N^j$ are the expectations under $\bQ_T^j$ and $\bQ_N^j$ respectively with
\begin{align}
\frac{d\mathbb{Q}_T^j}{d\bP^{N,T}_{\theta_j}}=&\exp\left(\epsilon_{\eta}^j\ln L(\theta_0,\theta_1,U_T^N)-T\mathcal{L}_T^j(\epsilon_{\eta}^j)\right),\label{eq:MeasureChange1}
\\
\frac{d\bQ_N^j}{d\bP^{N,T}_{\theta_j}}=&\exp\left(\widetilde{\epsilon}_{\eta}^j\ln L(\theta_0,\theta_1,U_T^N)-M\mathcal{L}_N^j(\widetilde{\epsilon}_{\eta}^j)\right).\label{eq:MeasureChange2}
\end{align}
By taking $\epsilon_{\eta}^j$ or $\widetilde{\epsilon}_{\eta}^j$ such that $M\mathcal{L}_j'(\epsilon_{\eta}^j)=\eta$ or $T\mathcal{L}_j'(\widetilde{\epsilon}_{\eta}^j)=\eta$, we got
\begin{align}
\epsilon_{\eta}^j=&\frac{(\theta_1^2-\theta_0^2)^2M^2-4\theta_j^2(-2\eta+(\theta_1-\theta_0)M)^2} {4(\theta_1^2-\theta_0^2)(-2\eta+(\theta_1-\theta_0)M)^2},\label{eq:epsilonT}
\\
\widetilde{\epsilon}_{\eta}^j=&\frac{(\theta_1^2-\theta_0^2)^2T^2-4\theta_j^2(-2\eta+(\theta_1-\theta_0)T)^2} {4(\theta_1^2-\theta_0^2)(-2\eta+(\theta_1-\theta_0)T)^2},\label{eq:epsilonN}
\end{align}
and then by direct computations we found that
\begin{align}
A_T^j=& \exp\left(-I_j(\eta)T\right) \exp\left[N\mathcal{H}_j(\epsilon_{\eta}^j)+\mathcal{R}_j(\epsilon_{\eta}^j)\right],\label{eq:ATVariate}
\\
\widetilde{A}_N^j=& \exp\left(-\widetilde{I}_j(\eta)M\right) \exp\left[N\mathcal{H}_j(\widetilde{\epsilon}_{\eta}^j)+\mathcal{R}_j (\widetilde{\epsilon}_{\eta}^j)\right],\label{eq:ANVariate}
\end{align}
where
\begin{align}
I_j(\eta)= -\frac{(4\theta_j\eta+(-1)^j(\theta_1-\theta_0)^2M)^2}{8(2\eta-(\theta_1-\theta_0)M) (\theta_1^2-\theta_0^2)},\qquad
\widetilde{I}_j(\eta)=-\frac{(4\theta_j\eta+(-1)^j(\theta_1-\theta_0)^2T)^2} {8(2\eta-(\theta_1-\theta_0)T)(\theta_1^2-\theta_0^2)}.\label{eq:Rate-I-TN}
\end{align}

Finally, also in \cite{CialencoXu2013a} we derived the large deviation principles for considered SPDEs
\begin{align}\label{LargeDev-Null}
&\lim_{T\to\infty}T^{-1}\ln \bP^{N,T}_{\theta_0}\left(T^{-1}\ln L(\theta_0,\theta_1,U_T^N)\ge\eta\right)=-I_0(\eta), &&\eta\in\left(-\frac{(\theta_1-\theta_0)^2}{4\theta_0}M,\frac{\theta_1-\theta_0}{2}M\right), \\
%\\
%&\lim_{T\to\infty}T^{-1}\ln \bP^{N,T}_{\theta_0}\left(T^{-1}\ln L(\theta_0,\theta_1,U_T^N)\le\eta\right)=-I_0(\eta), %&&\eta\in\left(-\infty,-\frac{(\theta_1-\theta_0)^2}{4\theta_0}M\right), \nonumber
%\end{align}
%and
%\begin{align}
\label{eq:LDP-Linkov}
&\lim_{T\to\infty}T^{-1}\ln \bP^{N,T}_{\theta_1}\left(T^{-1}\ln L(\theta_0,\theta_1,U_T^N)\geq\eta\right)=-I_1(\eta), &&\eta\in\left(\frac{(\theta_1-\theta_0)^2}{4\theta_1}M,\frac{\theta_1-\theta_0}{2}M\right),
%\\
%&\lim_{T\to\infty}T^{-1}\ln \bP^{N,T}_{\theta_1}\left(T^{-1}\ln L(\theta_0,\theta_1,U_T^N)\leq\eta\right)=-I_1(\eta), %&&\eta\in\left(-\infty,\frac{(\theta_1-\theta_0)^2}{4\theta_1}M\right), \nonumber
\end{align}

It should be mentioned that in \cite{CialencoXu2013a} the relations \eqref{eq:MomentLHRSplit}--\eqref{eq:LDP-Linkov} were derived only under the alternative hypothesis, $\theta=\theta_1$, however, the corresponding  results for  $\theta=\theta_0$ are obtained in a very similar manner.
The main difference is that $\theta_1$ in the PDE obtained by Feynman-Kac Formula is replaced by $\theta_0$, but the method of solving it remains of course the same.
We admit that some parts of these derivations may appear technically challenging, but nevertheless we felt unnecessary to mimic them here.

%We sketch the proof as follows. By following the idea in \cite{GapeevKuchler2008}, we used Feynman--Kac formula to derive a partial differential equation, whose solution is related to the cumulant generating function of $\ln L(\theta_0,\theta_1,U_T^N)$.
%Then we found the explicit formula of the cumulant generating function by solving the PDE. Finally, we applied a result in \cite{Linckov1999} and use some basic skills in limit theory to obtain \eqref{eq:LDP-Linkov}.
%This result actually tells that, with $c_\alpha$ properly defined, the Type~II error of likelihood ratio test goes exponentially fast to zero as $T\to\infty$, with an explicit rate function $I(\eta)$.

\section{The case of large times}\label{sec:largeT}
Throughout this section, we assume that  the number of Fourier modes $N$ is fixed.
Recall that without loss of generality we assume that $\theta_1>\theta_0$ (the obtained results are symmetric otherwise).
We still consider tests of the form $R_T=\{U_T^N: L(\theta_0,\theta_1,U_T^N)\ge c_\alpha(T)\}$, but for the sake of convenience we write them equivalently  as
\begin{align}\label{eq:RTtest-eta}
R_T=\{U_T^N: \ln L(\theta_0,\theta_1,U_T^N)\ge\eta T\},
\end{align}
where, unless specified, $\eta$ is an arbitrary number which may depend on $N$ and $T$.
Our goal is to find a proper expression for $\eta$ such that for $T$ larger than a certain number, the Type I and II errors are always smaller than a chosen threshold.
Clearly, we are looking for $\eta$ that is a  bounded function of $T$.
Using the results on large deviations from Section~\ref{sec:1-2}, we will first give an argument how to derive a proper expression of $\eta$, followed by main results and their detailed proofs.

Following the large deviation principle  \eqref{LargeDev-Null}, let us assume that $\eta$ is such that
\begin{align}\label{eq:Range-eta}
-\frac{(\theta_1-\theta_0)^2}{4\theta_0}M<\eta<\frac{\theta_1-\theta_0}{2}M.
\end{align}
Then, we have that $\epsilon_{\eta}^0>0$, and hence $B_T^0\leq 1.$
Consequently, in view of \eqref{eq:ABSplit-RT}, to get an upper bound for the Type~I error, it is enough to estimate $A_T^0$.
By \eqref{eq:ATVariate}, combined  with \eqref{LargeDev-Null}, we note that $\exp\left(-I_0(\eta)T\right)$ is the dominant term  of asymptotic expansion of Type~I error.
Since we have an explicit expression of the residual part $\exp\left[N\mathcal{H}_0(\epsilon_{\eta}^0)+\mathcal{R}_0(\epsilon_{\eta}^0)\right]$, this suggest that if we simply let the dominant part to be equal to the significance level $\alpha$, that is
\begin{align}\label{eq:Dom-Eq}
\exp\left(-I_0(\eta)T\right)=\alpha,
\end{align}
we may be able to control the Type I error by a much simpler function.
In fact, by solving equation \eqref{eq:Dom-Eq}, that has two solutions, and since $\eta$ has to satisfy \eqref{eq:Range-eta}, we choose
\begin{align}\label{ErrTest-gamma}
\eta=-\frac{(\theta_1-\theta_0)^2}{4\theta_0}M+\frac{(\theta_1^2-\theta_0^2)\ln\alpha}{2\theta_0^2T}+ \frac{\theta_1^2-\theta_0^2}{2\theta_0^2}\sqrt{-\theta_0MT^{-1}\ln\alpha+T^{-2}\ln^2\alpha}.
\end{align}
Clearly $\eta$ is a bounded function of $T$. Moreover, $\eta$ indeed satisfies \eqref{eq:Range-eta}, a point made clear by \eqref{Deltagamma} below.

Next we present the first main result of this paper that shows how large $T$ has to be so that the Type~I error is smaller than a given tolerance level.
\begin{theorem}\label{Firsterr-Est}
Assume that the test statistics has the form
\begin{align*}
R^0_T=\left\{U^N_T: \ln L(\theta_0,\theta_1,U_T^N)\ge\eta T\right\},
\end{align*}
where $\eta$ is given by \eqref{ErrTest-gamma}. If
\begin{align}\label{Req-1}
T\ge\max\left\{-\frac{256\theta_0\ln\alpha}{(\theta_1-\theta_0)^2M},-\frac{16\ln\alpha}{\theta_0 M}, -\frac{16(1+\varrho)^2\theta_0(\theta_1-\theta_0)^2(N+1)^2\ln\alpha}{\varrho^2(\theta_1+\theta_0)^4M}\right\},
\end{align}
then the Type~I error has the following bound estimate
\begin{align*}
\bP^{N,T}_{\theta_0}\left(R^0_T\right)\le(1+\varrho)\alpha,
\end{align*}
where $\varrho$ denotes a given threshold of error tolerance\footnote{Generally expected to be small, say less than 10\%. Smaller $\varrho$ will yield larger $T$, and the final choice is left to the observer.}.
\end{theorem}
\begin{proof}
Let us consider
\begin{align}\label{Deltagamma}
\Delta\eta:=\eta+\frac{(\theta_1-\theta_0)^2}{4\theta_0}M
=&\frac{\theta_1^2-\theta_0^2}{\theta_0^2}\frac{-\theta_0MT^{-1}\ln\alpha}{-T^{-1}\ln\alpha+\sqrt{-\theta_0MT^{-1}\ln\alpha +T^{-2}\ln^2\alpha}}\notag
\\
\le&(\theta_1^2-\theta_0^2)\sqrt{-\theta_0^{-3}M\ln\alpha}T^{-1/2}.
\end{align}
Note that  $\Delta\eta>0$, which implies that $\eta >  -(\theta_1-\theta_0)^2M/4\theta-0$. Moreover, since $\Delta \eta \to 0$, as $T\to\infty$, we also have that $\eta<(\theta_1-\theta_0)M/2$, for sufficiently large $T$, and hence \eqref{eq:Range-eta} is satisfied.

Substituting \eqref{ErrTest-gamma} into \eqref{eq:epsilonT}, by direct evaluations, we deduce
\begin{align}\label{EpsGamma-Null-1}
\epsilon_{\eta}^0=&\frac{\theta_0(\theta_1^2-\theta_0^2)M\Delta\eta-\theta_0^2\Delta\eta^2} {(\theta_1^2-\theta_0^2)((\theta_1^2-\theta_0^2)M/(2\theta_0)-\Delta\eta)^2}\le\frac{\theta_0M\Delta\eta} {((\theta_1^2-\theta_0^2)M/(2\theta_0)-\Delta\eta)^2}.
\end{align}
By \eqref{Deltagamma} and \eqref{EpsGamma-Null-1}, we conclude that, if
\begin{align}\label{condition-1}
(\theta_1^2-\theta_0^2)\sqrt{-\theta_0^{-3}M\ln\alpha}T^{-1/2}\le(\theta_1^2-\theta_0^2)M/(4\theta_0),
\end{align}
then have the following estimate
\begin{align}\label{Epsgamma-Est}
0<\epsilon_{\eta}^0\le\frac{16\theta_0^3\Delta\eta} {(\theta_1^2-\theta_0^2)^2M}\le\frac{16\sqrt{-\theta_0^3\ln\alpha}} {(\theta_1^2-\theta_0^2)\sqrt{M}}T^{-1/2}.
\end{align}
A straightforward inspection of the derivative of $\mathcal{D}_0(\epsilon)$ implies that $\mathcal{D}_0(\epsilon)$ decreases for $\epsilon<\frac{\theta_0}{\theta_1+\theta_0}$,  and goes to 1, as $\epsilon\to0+$. Thus, using \eqref{Epsgamma-Est}, if \begin{align}\label{condition-2}
\frac{16\sqrt{-\theta_0^3\ln\alpha}} {(\theta_1^2-\theta_0^2)\sqrt{M}}T^{-1/2}<\frac{\theta_0}{\theta_1+\theta_0},
\end{align}
then we can guarantee that $0<\mathcal{D}_0(\epsilon_{\eta}^0)<1$. From here, under assumption that  \eqref{condition-1} and \eqref{condition-2} hold true, we have
\begin{align}\label{Estimate-1}
\exp\left[\mathcal{R}_0(\epsilon_{\eta}^0)\right]=\prod_{k=1}^N\left(1+\frac{1-\mathcal{D}_0(\epsilon_{\eta}^0)}{1+\mathcal{D}_0(\epsilon_{\eta}^0)} \exp\left(-2\lambda_k^{2\beta}T\sqrt{\theta_0^2+(\theta_1^2-\theta_0^2)\epsilon_{\eta}^0}\right)\right)^{-1/2}<1.
\end{align}
Due to the fact that $\sqrt{1+x}<1+x/2$, we get
\begin{align*}
\mathcal{D}_0(\epsilon_{\eta}^0)\ge\frac{\theta_0+(\theta_1-\theta_0)\epsilon_{\eta}^0}{\theta_0+(\theta_1^2-\theta_0^2)\epsilon_{\eta}^0/(2\theta_0)}. \end{align*}
Therefore, under \eqref{condition-1} and \eqref{condition-2}, we obtain
\begin{align*}
\mathcal{D}_0(\epsilon_{\eta}^0)-1\ge&-\frac{(\theta_1-\theta_0)^2}{2\theta_0 \left(\theta_0+(\theta_1^2-\theta_0^2)\epsilon_{\eta}^0/(2\theta_0)\right)} \epsilon_{\eta}^0
\ge -\frac{(\theta_1-\theta_0)^2}{\theta_0 \left(\theta_1+\theta_0\right)} \epsilon_{\eta}^0 \\
\ge&-\frac{16(\theta_1-\theta_0)\sqrt{-\theta_0\ln\alpha}} {\left(\theta_1+\theta_0\right)^2\sqrt{M}}T^{-1/2}.
\end{align*}
From the above, and by means of Bernoulli inequality,  we continue
\begin{align}\label{Estimate-2}
\exp\left[N\mathcal{H}_0(\epsilon_{\eta}^0)\right]=&\left(1+\frac{1}{2}\left(\mathcal{D}_0(\epsilon_{\eta}^0)-1\right)\right)^{-N/2}\notag \\
\le&\left(1+\frac{1}{2}\left(\mathcal{D}_0(\epsilon_{\eta}^0)-1\right)\right)^{-\lfloor(N+1)/2\rfloor}\notag
\\
\le&\left(1+\frac{\lfloor(N+1)/2\rfloor}{2}\left(\mathcal{D}_0(\epsilon_{\eta}^0)-1\right)\right)^{-1}\notag
\\
\le&\left(1-\frac{4(N+1)(\theta_1-\theta_0)\sqrt{-\theta_0\ln\alpha}} {\left(\theta_1+\theta_0\right)^2\sqrt{M}}T^{-1/2}\right)^{-1}.
\end{align}
Note that the above inequalities hold true if all the terms in the parenthesis are positive, for which is enough to assume that
\begin{align}\label{extracond}
\frac{4(N+1)(\theta_1-\theta_0)\sqrt{-\theta_0\ln\alpha}} {\left(\theta_1+\theta_0\right)^2\sqrt{M}}T^{-1/2}<1.
\end{align}
Recall that $\epsilon_{\eta}^0>0$, and hence $B_T^0\leq 1$.
Using \eqref{eq:ABSplit-RT} and \eqref{eq:Dom-Eq}, combined with \eqref{Estimate-1} and \eqref{Estimate-2}, we conclude that
\begin{align*}
\bP^{N,T}_{\theta_0}\left(R^0_T\right)=A_T^0B_T^0\le \alpha\left(1-\frac{4(N+1)(\theta_1-\theta_0)\sqrt{-\theta_0\ln\alpha}} {\left(\theta_1+\theta_0\right)^2\sqrt{M}}T^{-1/2}\right)^{-1}.
\end{align*}
Thus, in order to make the Type I error to satisfy the desire upper bound $\bP^{N,T}_{\theta_0}\left(R^0_T\right)\le(1+\varrho)\alpha$, it is sufficient to require that \begin{align}\label{condition-3}
T\ge-\frac{16(1+\varrho)^2\theta_0(\theta_1-\theta_0)^2(N+1)^2\ln\alpha}{\varrho^2(\theta_1+\theta_0)^4M},
\end{align}
under assumption that \eqref{condition-1}, \eqref{condition-2} and \eqref{extracond} hold true, which is satisfied due to original assumption \eqref{Req-1}.
This concludes the proof.

\end{proof}

Next we will study the estimation of Type~II error, as time $T$ goes to infinity.

\begin{theorem}\label{Seconderr-Est}
Assume that the test $R^0_T$ is given as in Theorem \ref{Firsterr-Est}.
If
\begin{align}\label{Req-2}
T\ge\max\left\{-\frac{16(\theta_1^2+16\theta_0^2)\ln\alpha}{\theta_0(\theta_1-\theta_0)^2M},-\frac{16\ln\alpha}{\theta_0 M}, -\frac{16(1+\varrho)^2\theta_0(\theta_1-\theta_0)^2(N+1)^2\ln\alpha}{\varrho^2(\theta_1+\theta_0)^4M}\right\},
\end{align}
then the Type~II error admits the following upper bound estimate
\begin{align}\label{SecondEst}
1-\bP^{N,T}_{\theta_1}\left(R^0_T\right)\le(1+\varrho)\exp\left(-\frac{(\theta_1-\theta_0)^2}{16\theta_0^2}MT\right).
\end{align}
\end{theorem}

\begin{proof}
Let $\eta$ be as in \eqref{ErrTest-gamma}. By direct evaluations, one can show that
\begin{align*}
\mathcal{H}_1(\epsilon_{\eta}^1)=\mathcal{H}_0(\epsilon_{\eta}^0),\qquad \mathcal{R}_1(\epsilon_{\eta}^1)=\mathcal{R}_0(\epsilon_{\eta}^0).
\end{align*}
Recall that, from the previous theorem, assuming that \eqref{Req-1} holds true, we have that
\begin{align}\label{est-1}
\exp\left[N\mathcal{H}_1(\epsilon_{\eta}^1)+\mathcal{R}_1(\epsilon_{\eta}^1)\right]= \exp\left[N\mathcal{H}_0(\epsilon_{\eta}^0)+\mathcal{R}_0(\epsilon_{\eta}^0)\right]\le1+\varrho.
\end{align}
In view of  \eqref{Deltagamma} and \eqref{eq:Rate-I-TN}, if we further require that
\begin{align}\label{condition-4}
(\theta_1^2-\theta_0^2)\sqrt{-\theta_0^{-3}M\ln\alpha}T^{-1/2}\le\frac{(\theta_1^2-\theta_0^2)(\theta_1-\theta_0)} {4\theta_0\theta_1}M,
\end{align}
it can be easily deduced that
\begin{align}\label{est-2}
\exp\left(-I_1(\eta)T\right)\le\exp\left(-\frac{(\theta_1-\theta_0)^2}{16\theta_0^2}MT\right).
\end{align}
By \eqref{Epsgamma-Est}, assuming that \eqref{condition-2} holds true, we also have that
\begin{align*}
\epsilon_{\eta}^1=\epsilon_{\eta}^0-1<\frac{\theta_0}{\theta_1+\theta_0}-1<0,
\end{align*}
and hence
\begin{align}\label{est-3}
B_T^1=\bE_T^1\left(\exp\left[-\epsilon_{\eta}^1(\ln L(\theta_0,\theta_1,U_T^N)-\eta T)\right]\1_{\{\ln L(\theta_0,\theta_1,U_T^N)\le\eta T\}}\right)<1.
\end{align}
Note that \eqref{eq:ABSplit-RT}-\eqref{eq:ATVariate} imply that
\begin{align*}
1-\bP^{N,T}_{\theta_1}\left(R^0_T\right)=& \bP^{N,T}_{\theta_1}\left(\ln L(\theta_0,\theta_1,U_T^N)\le\eta T\right)=A_T^1B_T^1
\\
=& \exp\left(-I_1(\eta)T\right) \exp\left[N\mathcal{H}_1(\epsilon_{\eta}^1)+\mathcal{R}_1(\epsilon_{\eta}^1)\right]B_T^1.
\end{align*}
Therefore, \eqref{SecondEst} follows from \eqref{est-1}, \eqref{est-2} and \eqref{est-3}, under assumption that \eqref{Req-1} and \eqref{condition-4} are satisfied,
which is guaranteed by \eqref{Req-2}.
This finishes the proof.
\end{proof}

%%%%%%%%%%%%%%%%%%%%%%%%%%%%%%%%%%%%%%%%%%%%%%%%%%%%%%%%%%%%%%%%%%%%%%%%%%
%%%%%%%%%%%%%%%%%%%%%%%%%%%%%%%%%%%%%%%%%%%%%%%%%%%%%%%%%%%%%%%%%%%%%%%%%%

\section{The case of large number of Fourier modes}\label{sec:largeN}
In this section we study the error estimates for the case of large number of Fourier modes $N$, while the time horizon $T$ is fixed.
The key ideas and the method itself are similar to those developed in the previous section.
We consider tests of the form
\begin{align}\label{eq:RNtest-zeta}
R_N=\{U_T^N: \ln L(\theta_0,\theta_1,U_T^N)\ge\zeta M\},
\end{align}
where $\zeta$ is some number depending on $N$ and $T$, and where as before $M:=\sum_{k=1}^N\lambda_k^{2\beta}$.
The goal is to find $\zeta$, as a bounded function of $N$, that will allow to controll the statistical errors when the number of Fourier modes $N$ is large.

Similarly to $T$-part, for $\zeta>-\frac{(\theta_1-\theta_0)^2}{4\theta_0}T$, we have that $\widetilde{\epsilon}_{\zeta}^0>0$, and hence $\widetilde{B}_N^0\le1$.
Thus, it is enough to estimate $\widetilde{A}_N^0$, and by the same reasons as in Section~\ref{sec:largeT}, we let $\exp\left(-\widetilde{I}_0(\zeta)M\right)=\alpha$, and derive that the natural candidate for $\zeta$ has the following form
\begin{align}\label{ErrTest-eta-N}
\zeta=-\frac{(\theta_1-\theta_0)^2}{4\theta_0}T+\frac{(\theta_1^2-\theta_0^2)\ln\alpha}{2\theta_0^2M}+ \frac{\theta_1^2-\theta_0^2}{2\theta_0^2}\sqrt{-\theta_0TM^{-1}\ln\alpha+M^{-2}\ln^2\alpha}.
\end{align}

Next we provide the result on how large $N$ should be (for a fixed $T$) to guarantee that Type~I and Type~II errors are smaller than a given tolerance level.
\begin{theorem}\label{Firsterr-Est-N}
Consider the test
\begin{align*}
R^0_N=\left\{U^N_T: \ln L(\theta_0,\theta_1,U_T^N)\ge\zeta M\right\},
\end{align*}
where $\zeta$ is given by \eqref{ErrTest-eta-N}.
\begin{enumerate}[(i)]
\item If
\begin{align}\label{Req-1-N}
M\ge -\frac{16\ln\alpha}{\theta_0T}\max\left\{\frac{16\theta_0^2}{(\theta_1-\theta_0)^2},1\right\} \quad\textrm{and}\quad \frac{M}{(N+1)^2}\ge -\frac{16(1+\varrho)^2\theta_0(\theta_1-\theta_0)^2\ln\alpha}{\varrho^2(\theta_1+\theta_0)^4T},
\end{align}
then the Type~I error has the following upper bound estimate
\begin{align}
\bP^{N,T}_{\theta_0}\left(R^0_N\right)\le(1+\varrho)\alpha,
\end{align}
where $\varrho$ denotes a given threshold of error tolerance.

\item If
\begin{align}\label{Req-2-N}
M\ge -\frac{16\ln\alpha}{\theta_0 T} \max\left\{\frac{(\theta_1^2+16\theta_0^2)}{(\theta_1-\theta_0)^2},1\right\} \quad\textrm{and}\quad \frac{M}{(N+1)^2}\ge -\frac{16(1+\varrho)^2\theta_0(\theta_1-\theta_0)^2\ln\alpha}{\varrho^2(\theta_1+\theta_0)^4T},
\end{align}
we have the following estimate for Type II error
\begin{align}\label{SecondEst-N}
1-\bP^{N,T}_{\theta_1}\left(R^0_N\right)\le(1+\varrho)\exp\left(-\frac{(\theta_1-\theta_0)^2}{16\theta_0^2}MT\right).
\end{align}
\end{enumerate}

\end{theorem}
The proof is similar\footnote{For most of the derivations one just needs to `exchange $T$ with $M$.' The results are, in a sense, symmetric with respect to $T$ and $M$. In \eqref{Req-1-N} and \eqref{Req-2-N} we separate the conditions for $N$ into two inequalities, since we want to place all the terms related to $N$ on the left side of the inequalities.} to the proofs of Theorem~\ref{Firsterr-Est} and Theorem~\ref{Seconderr-Est}, and we omit it here\footnote{We need to point out that sometimes we may not be able to find $N$ such that the conditions \eqref{Req-1-N} and \eqref{Req-2-N} are satisfied. For example, if $\beta/d\le 1/2$ then $M/(N+1)^2$ is bounded for all $N\in\bN$, and if its bound is smaller than the right hand side of the second inequality in \eqref{Req-1-N} and \eqref{Req-2-N}, then the conditions \eqref{Req-1-N} and \eqref{Req-2-N} fail for all $N$. However, for $\beta/d\le 1/2$ we might still be able to control the Type~I and Type~II errors by finite $N$, which requires a more technical proof and is deferred to future study.}.

\section{Numerical Experiments}\label{sec:NumericalSim}
In this section we give a simple illustration of theoretical results from previous sections by means of numerical simulations.
Besides showing the behavior of Type~I and Type~II errors for the test $R^0$ proposed in this paper, we will also display the simulation results for $R^\sharp$ test mentioned in Section~\ref{sec:1-1} and discussed in \cite{CialencoXu2013a}.
We start with description of the numerical scheme used for simulation of trajectories of the solution (more precisely of the Fourier modes), and provide a brief argument on the error estimates of the corresponding Monte Carlo experiments associated with this scheme.
In the second part of the section, we focus on numerical interpretation of the theoretical results obtained in Sections~\ref{sec:largeT} and \ref{sec:largeN}.

%Please note that we only provide a brief argument on the error estimate for the numerical scheme, because we present the simulation results merely to help the readers better understand how the idea of error control works.

%In the following we briefly describe our algorithm for discretization with some mathematical expressions.

We use the standard Euler-Maruyama scheme\footnote{Of course many other discretizations of equation \eqref{eq:mainSPDE} can be chosen, such as implicit Euler scheme, or exponential Euler scheme, that can be computationally more efficient; cf. the monograph \cite{JentzenKloeden2011Book}.}
to numerically approximate the trajectories of the Fourier modes $u_k(t)$ given by equation \eqref{eq:OU-Fourier}, and we apply Monte Carlo method to estimate the Type~I and Type~II errors. We partition the time interval $[0,T]$ into $n$ equality spaced time intervals $0=t_0<t_1<\ldots<t_n=T$, with $\Delta T=T/n = t_{i}-t_{i-1}$, for $1\le i\le n$.
Let $m$ denote the number of trials in the Monte Carlo experiment of each Fourier mode.
 Assume that $u_k^j(t_i)$ is the true value of the $k$-th Fourier mode at time $t_i$ of the $j$-th trial in Monte Carlo simulation.
Then, for every $1\leq k\leq N$, $1\leq j \leq m$, we approximate $u_k^j(t_i)$ according to the following recursion formula
\begin{align}\label{eq:OU-recursion}
\widetilde{u}_k^j(t_{i}) = \widetilde{u}_k^j(t_{i-1}) - \theta \lambda_k^{2\beta} \widetilde{u}_k^j(t_{i-1}) \Delta T + \sigma \lambda_k^{-\gamma} \xi_{k,i}^j,\qquad \widetilde{u}_k^j(t_0)= u_k(0), \quad 1\leq i \leq n.
\end{align}
where $\xi_{k,i}^j$ are i.i.d. Gaussian random variables with zero mean and variance $\Delta T$.
%Thus, $\widetilde{u}_k^j(t_{i})$ being an approximation of $u_k^j(t_i)$ for all $1\le i\le n$, $1\le j\le m$, $1\le k\le N$,
In what follows, we will investigate how to approximate the Type~I and Type~II errors of $R^0$ test using $\widetilde{u}_k^j(t_{i})$'s, and how the numerical errors are related to $n$, $m$, $T$ and $N$.

Throughout this section we consider equation \eqref{eq:mainSPDE}, and consequently \eqref{eq:OU-recursion}, with $\beta=1$, in one dimensional space $d=1$, with the random forcing term being the space-time white noise $\gamma=0, \ \sigma=1$.
We also assume that the spacial domain $G=[0,\pi]$ and the initial value $U_0=0$. In this case $\lambda_k=k, \ k\in\bN$. We fix the parameter of interest to be $\theta_0=0.1$ and $\theta_1=0.2$.
The general case is treated analogously, the authors feel that a complete and detailed analysis of the numerical results are beyond the scope of the current publication.
The numerical simulations presented here are intended to show a simple analysis of the proposed methods.
We performed simulations for other sets of parameters, and the numerical results were in concordance with the theoretical ones.
For example, for the case of large times, if one increases $N$, then the statistical errors are reaching the threshold for smaller values of $T$ - more information improves the rate of convergence.
Similarly, increasing $T$ for the case of asymptotics in $N$, one needs to take fewer Fourier modes to bypass the threshold of the statistical errors.
Different ranges and magnitudes of the parameter of interest $\theta$ were considered, and the outcomes are similar to those presented below.
All simulations and computations are done in MATLAB and the source code is available from the authors upon request.

%First, we described into details the numerical procedure used to approximate the path of the Fourier modes, as well as the Monte Carlo experiment for approximation of the statistical errors.

\subsection{Description and analysis of the numerical experiments}\label{sec:NumSchemeT}

Throughout $C$ denotes a constant, whose value may vary from line to line, and whenever the formulas or results are indexed by $j$, we mean that they hold true for all $1\le j\le m$. Using \eqref{eq:RadonNikodymUn}, and by It\=o's formula, we get
\begin{align}\label{eq:heuristT}
\bP^{N,T}_{\theta_0}& (R_T^0) = \bP^{N,T}_{\theta_0} (\ln L(\theta_0,\theta_1,U_T^N)\ge \eta T) \notag
\\
=&\bP^{N,T}_{\theta_0}\left(-\sum_{k=1}^N\lambda_k^{2\beta+2\gamma}\left(\int_0^Tu_k(t)du_k(t) +\frac{\theta_1+\theta_0}{2\theta_0}\int_0^Tu_k\left(\sigma\lambda_k^{-\gamma}dw_k-du_k\right)\right)\ge \frac{\sigma^{2}\eta T}{\theta_1-\theta_0}\right)\notag
\\
%=&\bP^{N,T}_{\theta_0}\left(\sum_{k=1}^N\lambda_k^{2\beta+2\gamma}\left((\theta_1-\theta_0)\int_0^Tu_kdu_k-(\theta_1+\theta_0)\sigma\lambda_k^{-\gamma}\int_0^Tu_kdw_k\right)\ge \frac{2\theta_0\sigma^{2}\ln c^*_{\alpha}}{\theta_1-\theta_0}\right)\notag
%\\
=&\bP^{N,T}_{\theta_0}\left(\sum_{k=1}^N\lambda_k^{2\beta+2\gamma}\left(\frac{\theta_1-\theta_0}{2}\left(u_k^2(T)-\sigma^2\lambda_k^{-2\gamma}T\right)-(\theta_1+\theta_0)\sigma\lambda_k^{-\gamma}\int_0^Tu_kdw_k\right)\ge \frac{2\theta_0\sigma^{2}\eta T}{\theta_1-\theta_0}\right)\notag
\\
%=&\bP^{N,T}_{\theta_0}\left(\sum_{k=1}^N\frac{\lambda_k^{2\beta+2\gamma}u_k^2(T)}{\sigma^2T}-\frac{2(\theta_1+\theta_0)}{(\theta_1-\theta_0)\sigma T}\sum_{k=1}^N\lambda_k^{2\beta+\gamma}\int_0^Tu_kdw_k\ge \frac{4\theta_0\ln c^*_{\alpha}}{(\theta_1-\theta_0)^2T}+M\right)\notag
%\\
= &\bP^{N,T}_{\theta_0}\left(\frac{(\theta_1-\theta_0)}{2\sigma (\theta_1+\theta_0)\sqrt{T}}X_T-Y_T/\sqrt{T}\ge \frac{2\theta_0 \sigma \Delta\eta }{\theta_1^2-\theta_0^2 }\sqrt{T}\right),
\end{align}
where $\eta$ and $\Delta \eta$ are given by \eqref{ErrTest-gamma} and \eqref{Deltagamma} respectively, and
\begin{align*}
X_T:=\sum_{k=1}^N\lambda_k^{2\beta+2\gamma}u_k^2(T),\qquad Y_T:=\sum_{k=1}^N\lambda_k^{2\beta+\gamma}\int_0^Tu_kdw_k.
\end{align*}
We approximate $X_T$ and $Y_T$ as follows
\begin{align*}
\widetilde{X}_{n,T}^j:=\sum_{k=1}^N\lambda_k^{2\beta+2\gamma}\widetilde{u}_k^j(t_{n})^2,\qquad
\widetilde{Y}_{n,T}^j:=\sum_{k=1}^N\lambda_k^{2\beta+\gamma}\sum_{i=1}^n \widetilde{u}_k^j(t_{i-1}) \xi_{k,i}^j.
\end{align*}
Define
\begin{align*}
\widetilde{R}_{n,T}^{0,j}:=\left\{\frac{(\theta_1-\theta_0)}{2\sigma (\theta_1+\theta_0)\sqrt{T}}\widetilde{X}_{n,T}^j-\widetilde{Y}_{n,T}^j/\sqrt{T} \ge \frac{2\theta_0 \sigma \Delta\eta }{\theta_1^2-\theta_0^2 }\sqrt{T}\right\}.
\end{align*}
Then, naturally, the approximation of $\bP^{N,T}_{\theta_0}(R_T^0)$ is given by
\begin{align}\label{eq:NumEst-PT0}
\widetilde{\cP}_{\theta_0}^{m,n,N,T}(R_T^0):= \frac{1}{m}\sum_{j=1}^m \1_{\widetilde{R}_{n,T}^{0,j}}.
\end{align}
Following \cite[Chapter 8]{Bishwal2008}, one can prove that
\begin{align}\label{eq:XY-T-L2dist}
\bE\left|\left(Y_T-\widetilde{Y}_{n,T}^j\right)/\sqrt{T}\right|^2=O(\Delta T),\qquad
\bE\left|X_T-\widetilde{X}_{n,T}^j\right|= O(\Delta T).
\end{align}
Consequently, for any $\epsilon>0$, we have
\begin{align*}
\bP^{N,T}_{\theta_0}\left(\widetilde{R}_{n,T}^{0,j}\right)\le& \bP^{N,T}_{\theta_0}\left(\frac{(\theta_1-\theta_0)}{2\sigma (\theta_1+\theta_0)\sqrt{T}}X_T-Y_T/\sqrt{T}\ge \frac{2\theta_0 \sigma \Delta\eta }{\theta_1^2-\theta_0^2 }\sqrt{T}-\epsilon\right)
\\
&+\bP^{N,T}_{\theta_0}\left(\left|Y_T-\widetilde{Y}_{n,T}^j\right|/\sqrt{T}\ge\epsilon/2\right)+ \bP^{N,T}_{\theta_0}\left(\frac{(\theta_1-\theta_0)}{2\sigma (\theta_1+\theta_0)\sqrt{T}}\left|X_T-\widetilde{X}_{n,T}^j\right|\ge\epsilon/2\right).
\end{align*}
According to \cite[Lemma~3.13]{CialencoXu2013a},  for large enough $T$, the following estimate holds true
\begin{align*}
\bP^{N,T}_{\theta_0}\left(\frac{(\theta_1-\theta_0)}{2\sigma (\theta_1+\theta_0)\sqrt{T}}X_T-Y_T/\sqrt{T}\ge \frac{2\theta_0 \sigma \Delta\eta }{\theta_1^2-\theta_0^2 }\sqrt{T}-\epsilon\right)\le \bP^{N,T}_{\theta_0} (R_T^0)(1+C\epsilon).
\end{align*}
By the above results, and Chebyshev inequality, we conclude that
\begin{align*}
\bP^{N,T}_{\theta_0}\left(\widetilde{R}_{n,T}^{0,j}\right)\le\bP^{N,T}_{\theta_0} (R_T^0)(1+C\epsilon)+C\epsilon^{-1}\bE\left|X_T-\widetilde{X}_{n,T}^j\right|/\sqrt{T}+ C\epsilon^{-2}\bE\left|\left(Y_T-\widetilde{Y}_{n,T}^j\right)/\sqrt{T}\right|^2.
\end{align*}
Similarly, we have that
\begin{align*}
\bP^{N,T}_{\theta_0}\left(\widetilde{R}_{n,T}^{0,j}\right)\ge \bP^{N,T}_{\theta_0} (R_T^0)(1-C\epsilon)-C\epsilon^{-1}\bE\left|X_T-\widetilde{X}_{n,T}^j\right|/\sqrt{T}- C\epsilon^{-2}\bE\left|\left(Y_T-\widetilde{Y}_{n,T}^j\right)/\sqrt{T}\right|^2.
\end{align*}
Combining the above two inequalities, we obtain that, for any $\epsilon>0$,
\begin{align*}
\left| \bP^{N,T}_{\theta_0}\left(\widetilde{R}_{n,T}^{0,j}\right)-\bP^{N,T}_{\theta_0} (R_T^0) \right|\le& C \epsilon\bP^{N,T}_{\theta_0} (R_T^0)+ C\epsilon^{-1}\bE\left|X_T-\widetilde{X}_{n,T}^j\right|/\sqrt{T}
\\
&+ C\epsilon^{-2}\bE\left|\left(Y_T-\widetilde{Y}_{n,T}^j\right)/\sqrt{T}\right|^2.
\end{align*}
This implies that
\begin{align}\label{eq:ErrOrder-RTn}
\left| \bP^{N,T}_{\theta_0}\left(\widetilde{R}_{n,T}^{0,j}\right)-\bP^{N,T}_{\theta_0} (R_T^0) \right|\le C_0\Delta T^{1/3},
\end{align}
where $C_0$ is a constant, which is small as long as $\bP^{N,T}_{\theta_0} (R_T^0)$ is small.
%However, based on the simulation result in Table \ref{table:ConvergOrder-T}, we believe that the order of convergence in \eqref{eq:ErrOrder-RTn} is much better, that is
%\begin{align*}
%\left| \bP^{N,T}_{\theta_0}\left(\widetilde{R}_{n,T}^{0,j}\right)-\bP^{N,T}_{\theta_0} (R_T^0) \right|\le C\Delta T.
%\end{align*}
It is straightforward to check that for large $T$
\begin{align*}
\textrm{Var}\left(\frac{(\theta_1-\theta_0)}{2\sigma (\theta_1+\theta_0)\sqrt{T}}X_T-Y_T/\sqrt{T}\right)\le C,
\end{align*}
where $C$ is a constant independent of $T$. From here and using \eqref{eq:XY-T-L2dist},  one can also show that
\begin{align*}
\textrm{Var}\left(\frac{(\theta_1-\theta_0)} {2\sigma (\theta_1+\theta_0)\sqrt{T}}\widetilde{X}_{n,T}^j-\widetilde{Y}_{n,T}^j/\sqrt{T}\right)=\textrm{Var}\left(\frac{(\theta_1-\theta_0)} {2\sigma (\theta_1+\theta_0)\sqrt{T}}X_T-Y_T/\sqrt{T}\right)+ O(\Delta T).
\end{align*}
This implies that the error of Monte Carlo simulations can be controlled by $m^{-1/2}$ uniformly with respect to $T$ and $n$. Therefore, we have the following error estimate
\begin{align}\label{eq:ErrEst-PmnNT}
\left|\widetilde{\cP}_{\theta_0}^{m,n,N,T}(R_T^0)-\bP^{N,T}_{\theta_0} (R_T^0)\right|\le C_1\Delta T^{1/3} + C_2 m^{-1/2},
\end{align}
which holds true with high probability (confidence interval of the Monte Carlo experiment).
Here $C_1$ is a constant which depends on $\bP^{N,T}_{\theta_0} (R_T^0)$ (usually small), and $C_2$ is a constant which only depends on the confidence level of Monte Carlo simulations. Thus, the estimator $\widetilde{\cP}_{\theta_0}^{m,n,N,T}(R_T^0)$ can be made arbitrarily close to the true value of $\bP^{N,T}_{\theta_0} (R_T^0)$ with arbitrarily high probability, as long as we take small enough time step $\Delta T$ and large enough number of trials $m$ of Monte Carlo simulations.

To approximate the value of $\bP^{N,T}_{\theta_0} (R_T^\sharp)$, similarly to \eqref{eq:heuristT}, we obtain
\begin{align*}
\bP^{N,T}_{\theta_0}& (R_T^\sharp)= \bP^{N,T}_{\theta_0}\left(\frac{(\theta_1-\theta_0)}{2\sigma (\theta_1+\theta_0)\sqrt{T}}X_T-Y_T/\sqrt{T}\ge -\sigma q_\alpha\sqrt{M/2\theta_0}\right),
\end{align*}
and we define
\begin{align*}
\widetilde{R}_{n,T}^{\sharp,j}:=\left\{\frac{(\theta_1-\theta_0)}{2\sigma (\theta_1+\theta_0)\sqrt{T}}\widetilde{X}_{n,T}^j-\widetilde{Y}_{n,T}^j/\sqrt{T} \ge -\sigma q_\alpha\sqrt{M/2\theta_0}\right\}.
\end{align*}
Then, the approximation of $\bP^{N,T}_{\theta_0}(R_T^\sharp)$ is given by
\begin{align}\label{eq:NumEst-PTsharp}
\widetilde{\cP}_{\theta_0}^{m,n,N,T}(R_T^\sharp):= \frac{1}{m}\sum_{j=1}^m \1_{\widetilde{R}_{n,T}^{\sharp,j}}.
\end{align}
Following the same proof we obtain error estimates similar to \eqref{eq:ErrEst-PmnNT} for $R_T^\sharp$.

Next we will present some numerical results that validate relationship \eqref{eq:ErrEst-PmnNT}.
In Table \ref{table:ConvergOrder-T}, we list simulation results of \eqref{eq:NumEst-PT0} for various value of the time step $\Delta T$ (or number of time steps $n$), while keeping fixed time horizon $T=100$, number of Monte Carlo simulations $m=20,000$, and number of Fourier modes $N=3$.
For convenience, we present same results in graphical form, Figure~\ref{Graph-TypIerr-T}.

\begin{table}[H]
\centering
\caption{Type~I error for various time steps $\Delta T$ (or number of time steps $n$)} %; Theorem \ref{Firsterr-Est}}
\begin{tabular}{c c c c c c c c c c c c} % centered columns (5 columns)
\hline \\ [-1ex] %inserts double horizontal lines
$\Delta T$ & $1$ & $0.9$ & $0.8$ & $0.7$ & $0.6$ & $0.5$ & $0.4$ & $0.3$ & $0.2$ \\ [0.5ex]
\hline \\ [-2ex]
$n$ & $100$ & $111$ & $125$ & $143$ & $167$ & $200$ & $250$ & $333$ & $500$ \\ [0.5ex] % inserting body of the table
\hline \\ [-1ex]
\small{$\widetilde{\cP}_{\theta_0}^{m,n,N,T}(R_T^0)$} & 0.0475 & 0.0375 & 0.0342 & 0.0283 & 0.0239 & 0.0202 & 0.0165 & 0.0157 & 0.0129 \\ [1ex] % [1ex] adds vertical space
\hline \\ [-1ex]
\small{$\widetilde{\cP}_{\theta_0}^{m,n,N,T}(R_T^\sharp)$} & 0.0975 & 0.0897 & 0.0802 & 0.0746 & 0.0686 & 0.0620 & 0.0566 & 0.0515 & 0.0503 \\ [1ex] % [1ex] adds vertical space

\hline\hline \\ [-0.5ex] %inserts single line

$\Delta T$ & $0.1$ & $0.09$ & $0.08$ & $0.07$ & $0.06$ & $0.05$ & $0.04$ & $0.03$ & $0.02$ \\ [0.5ex]
\hline \\ [-2ex]
$n$ & $1000$ & $1111$ & $1250$ & $1429$ & $1667$ & $2000$ & $2500$ & $3333$ & $5000$ \\ [0.5ex] % inserting body of the table
\hline \\ [-1ex]
\small{$\widetilde{\cP}_{\theta_0}^{m,n,N,T}(R_T^0)$} & 0.0102 & 0.0111 & 0.0099 & 0.0101 & 0.0096 & 0.0108 & 0.0089 & 0.0078 & 0.0088 \\ [1ex] % [1ex] adds vertical space
\hline \\ [-1ex]
\small{$\widetilde{\cP}_{\theta_0}^{m,n,N,T}(R_T^\sharp)$} & 0.0453 & 0.0416 & 0.0443 & 0.0413 & 0.0428 & 0.0401 & 0.0421 & 0.0400 & 0.0385 \\ [1ex] % [1ex] adds vertical space

\hline\hline \\ [-2ex] %inserts single line
\end{tabular}
\\
\small{Other parameters: $m=2\times 10^4$, $\alpha=0.05$, $T=100$, $\theta_0=0.1$, $\theta_1=0.2$, \\ $N=3$,  $\varrho=0.1$, $d=\beta=\sigma=1, \ \gamma=0$.}
%\\
%\small{Other parametric setting: $\theta_0=0.1$, $\theta_1=0.2$, $\sigma=1$, \\$d=1$, $\beta=1$, $\gamma=0$, $N=1$, $\rho=0.1$, $\Delta T=2000$.}
\label{table:ConvergOrder-T}
\end{table}

%\begin{figure}
%\centering
%\includegraphics[height=10cm]{Fig1-05-16-2014-v1.jpg} %{1.jpg}
%\caption{Graphical Description of Table \ref{table:ConvergOrder-T}}
%\label{Graph-TypIerr-T}
%\end{figure}

\begin{figure}[h]
\centering
\includegraphics[width=1.2\linewidth]{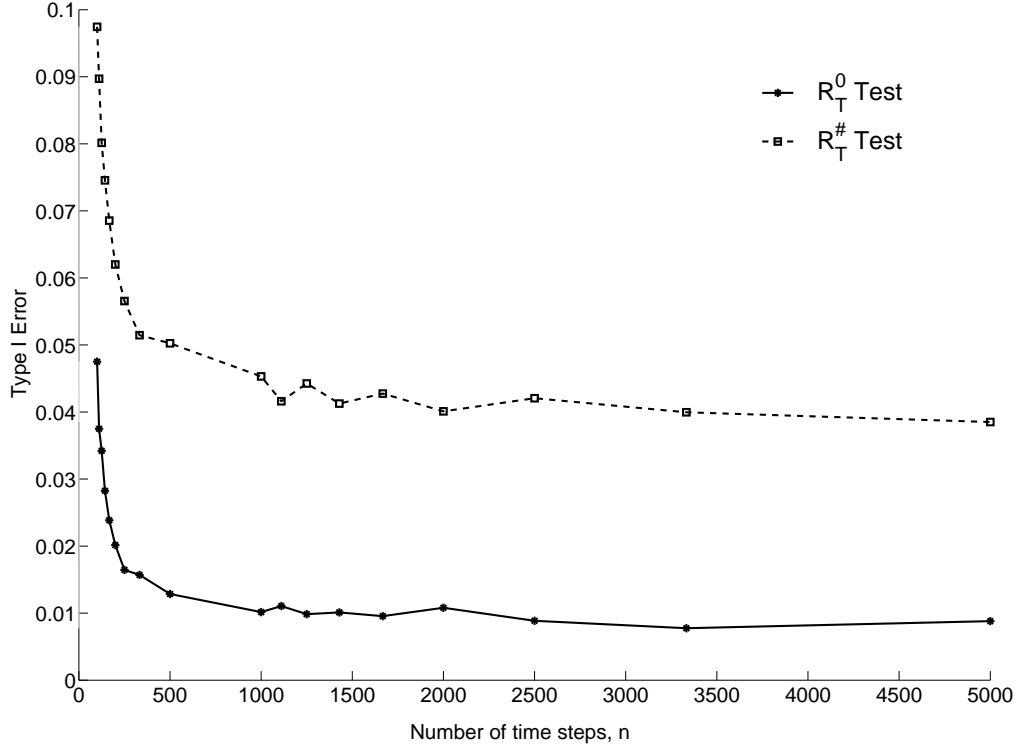}
\caption{Type~I error as a function of number of time steps $n$.
        \\ {\footnotesize Graphical interpretation of Table \ref{table:ConvergOrder-T}.}}
\label{Graph-TypIerr-T}
\end{figure}

As shown in Figure~\ref{Graph-TypIerr-T} the value of $\widetilde{\cP}_{\theta_0}^{m,n,N,T}(R_T^0)$, and respectively $\widetilde{\cP}_{\theta_0}^{m,n,N,T}(R_T^\sharp)$, rapidly decays  (approximatively up to the point when $n=5000$ or $\Delta T = 0.02$), and then it steadily approaches a certain `asymptotic level', which, as suggested by \eqref{eq:ErrEst-PmnNT}, shall be the true value of $\bP_{\theta_0}^{N,T}(R_{T}^0)$ (or $\bP_{\theta_0}^{N,T}(R_{T}^\sharp)$).
This assumes a reasonable large value of $m$, in our case $m=20,000$.
When $\Delta T$ gets smaller, we notice small fluctuations around that `asymptotic level', which are errors induced by the Monte Carlo method, and one can increase the number of trials to locate more precisely that true value. In our case the fluctuations are negligible comparative to the order of $\alpha$.
%Since for fixed $m$ the approximation error are controlled by $\Delta T$, we only need to increase $n$ proportionally to maintain the same degree of approximation when we take time $T$ larger.
%Sometimes it is not possible to determine the true value (or say, to locate the `asymptotic line') very accurately, especially when the true value is very small.
%In this case, while taking larger $m$ to enhance the accuracy, we increase $n$ more and more until the numerical value begins fluctuating up and down very slightly (which is the effect of Monte Carlo method) and will not make a distinct change any more (this is exactly how we plot Figure~\ref{Graph-TypIerr-T}).
%Then we stop increasing $n$ and take the instant numerical value as the true value of $\bP_{\theta_0}^{N,T}(R_{T}^0)$ (or $\bP_{\theta_0}^{N,T}(R_{T}^\sharp)$).
%It is true that this method will lead to inaccuracy, but the error is acceptable since it is already small enough compared to the order of $\alpha$ and will not affect demonstrating our results for `error control method' (as you will see, the computed values of $\bP_{\theta_0}^{N,T}(R_{T}^0)$ are always much smaller than $\alpha$ for different $\alpha$ and $T$ while the computed values of $\bP_{\theta_0}^{N,T}(R_{T}^\sharp)$ always fluctuate around $\alpha$).
%It could be possible that the error caused by Monte  Carlo simulation makes the `asymptotic line' slightly deviate from the real one, but this can happen only when almost all data points together deviate upwards or downwards, which is a very small probability event when we take $m$ large enough.

Now we fix the time horizon $T$, and vary the number of Fourier modes $N$.
Similarly to derivation of \eqref{eq:heuristT}, we have
\begin{align*}
\bP^{N,T}_{\theta_0}& (R_N^0)= \bP^{N,T}_{\theta_0}\left(\frac{(\theta_1-\theta_0)}{2\sigma (\theta_1+\theta_0)\sqrt{M}}X_T-Y_T/\sqrt{M}\ge \frac{2\theta_0 \sigma \Delta\zeta }{\theta_1^2-\theta_0^2 }\sqrt{M}\right),
\end{align*}
where
\begin{align*}
\Delta\zeta =\frac{(\theta_1^2-\theta_0^2)\ln\alpha}{2\theta_0^2M}+ \frac{\theta_1^2-\theta_0^2}{2\theta_0^2}\sqrt{-\theta_0TM^{-1}\ln\alpha+M^{-2}\ln^2\alpha}.
\end{align*}
Next, we define
\begin{align*}
\widetilde{R}_{n,N}^{0,j}:=\left\{\frac{(\theta_1-\theta_0)}{2\sigma (\theta_1+\theta_0)\sqrt{M}}\widetilde{X}_{n,T}^j-\widetilde{Y}_{n,T}^j/\sqrt{M} \ge \frac{2\theta_0 \sigma \Delta\zeta }{\theta_1^2-\theta_0^2 }\sqrt{M}\right\}.
\end{align*}
and approximate the probability $\bP^{N,T}_{\theta_0}(R_N^0)$ by
\begin{align}\label{eq:NumEst-PN0}
\widetilde{\cP}_{\theta_0}^{m,n,N,T}(R_N^0):= \frac{1}{m}\sum_{j=1}^m \1_{\widetilde{R}_{n,N}^{0,j}}.
\end{align}
One can prove\footnote{As usually, the case of large $N$ is more delicate and technically  challenging, comparative to the case of large times.
Apparently, \eqref{eq:XY-N-L2dist} holds true for some positive $\nu$. The sharpest value of $\nu$ is not relevant for this paper, and we defer the derivation of it to future study.}  that for some $\nu\ge 0$,
\begin{align}\label{eq:XY-N-L2dist}
\bE\left|\left(Y_T-\widetilde{Y}_{n,T}^j\right)/\sqrt{M}\right|^2=O(N^{\nu}/n),\qquad
\bE\left|X_T-\widetilde{X}_{n,T}^j\right|= O(N^{\nu}/n).
\end{align}
Following the same procedure as for large time asymptotics, we get
\begin{align}
\left|\widetilde{\cP}_{\theta_0}^{m,n,N,T}(R_N^0)-\bP^{N,T}_{\theta_0} (R_N^0)\right|\le C_1N^{\nu/3}n^{-1/3} + C_2 m^{-1/2},
\end{align}
where $C_1$ is a constant which depends on $\bP^{N,T}_{\theta_0} (R_N^0)$, and $C_2$ is a constant which depends on the confidence level of Monte Carlo experiment.

Similar results are obtained for the approximation of $\bP^{N,T}_{\theta_0} (R_N^\sharp)$ and the Type~II errors $\bP^{N,T}_{\theta_1} (R_N^0)$, $\bP^{N,T}_{\theta_1} (R_N^\sharp)$, $\bP^{N,T}_{\theta_1} (R_T^0)$ and $\bP^{N,T}_{\theta_1} (R_T^\sharp)$, and for brevity we will omit them here.

We conclude that the errors due to the numerical approximations considered above are negligible.
Hence, the  numerical methods  we propose are suitable for our purposes of computing the statistical errors of $R_T^0$, $R_T^\sharp$,  $R_N^0$ and $R_N^\sharp$ tests, and we  will use them for derivation of all numerical results from the next sections.

%, as well as for various number of Monte Carlo simulation $m$, but we them here.

\subsection{Numerical tests for large times}
We start with the case of large times $T$ and fixed $N$, and the results discussed in Section~\ref{sec:largeT}.
We take that $N=3$, i.e. we observe one path of the first three Fourier modes of the solution $u$ over some time interval $[0,T]$.
For convenience, we denote by $T_b^1$, and respectively $T_b^2$, the lower bound thresholds for $T$ from Theorem~\ref{Firsterr-Est}, relation \eqref{Req-1}, and respectively Theorem~\ref{Seconderr-Est}, relation \eqref{Req-2}.
%Then the approximation of the Log-likelihood ratio between $\theta_1$ and $\theta_0$ should be
%\begin{align*}
%L_{n,N,T}^j=-(\theta_1-\theta_0)\sum_{k=1}^{N} k^2\left(\sum_{i=1}^{n} u_k^j(t_{i-1})\left(u_k^j(t_i)-u_k^j(t_{i-1})\right)+ \frac{1}{2}(\theta_1+\theta_0)k^2\sum_{i=1}^{n} u_k^j(t_{i-1})^2 \Delta T\right).
%\end{align*}
In Table~\ref{table:VeriTh2.1alpha}, we list the Type~I error $\bP_{\theta_0}^{N,T}\left(R_{T}^0\right)$, along with corresponding values of $T_b^1$, for various values of $\alpha$.  Note that for all values of $\alpha$, the Type~I error is smaller than the threshold $(1+\varrho)\alpha$, and as expected, being  on conservative side.

\begin{table}[H]
\centering
\caption{
$T=T_b^1$ given by Theorem \ref{Firsterr-Est} and Type~I error for various $\alpha$.
%\small{$\theta_0=0.1$, $\theta_1=0.2$, $N=3$, $\rho=0.1$.}
}
\begin{tabular}{c c c c c} % centered columns (5 columns)
\hline \\ [-1ex] %inserts double horizontal lines
$\alpha$ & 0.1 & 0.05 & 0.01 & 0.005 \\ [0.5ex] % inserts table
%heading
\hline \\ [-1ex] % inserts single horizontal line
$T_b^1$ & 629 & 818 & 1258 & 1447 \\ [0.5ex] % inserting body of the table
\hline \\ [-1ex]
$\bP_{\theta_0}^{N,T}\left(R_{T}^0\right)$ &  0.021 &  0.010 & 0.0025 &  0.0015 \\ [1ex] % [1ex] adds vertical space

\hline\hline \\ [-2ex] %inserts single line
\end{tabular}
\\
\small{Other parameters: $\theta_0=0.1$, $\theta_1=0.2$, $N=3$, \\ $\varrho=0.1$, $d=\beta=\sigma=1, \ \gamma=0$.}
\label{table:VeriTh2.1alpha}
\end{table}

In Table~\ref{table:VeriTh2.1T} we show that for $T\geq T_b^1$, the error remains smaller than the chosen bound.
In fact, the Type~I error is decreasing as $T$ gets larger, with all other parameters fixed.

\begin{table}[H]
\centering
\caption{Type~I error for various $T\geq T_b^1$, with $T_b^1$ as in Theorem \ref{Firsterr-Est}}
\begin{tabular}{c c c c c c c} % centered columns (5 columns)
\hline \\ [-1ex] %inserts double horizontal lines
$T$ & $T_b^1$ & $T_b^1+ T_\delta$ & $T_b^1+ 2T_\delta$ & $T_b^1+ 3T_\delta$ & $T_b^1+ 4T_\delta$ & $T_b^1+ 5T_\delta$ \\ [0.5ex] % inserting body of the table
\hline \\ [-1ex]
$\bP_{\theta_0}^{N,T}\left(R_{T}^0\right)$ & 0.0100 & 0.0097 & 0.0105 & 0.0100 & 0.0105 & 0.0102 \\ [1ex] % [1ex] adds vertical space
\hline \\ [-1ex]
$\bP_{\theta_0}^{N,T}\left(R_{T}^\sharp\right)$ & 0.0540 & 0.0525 & 0.0505 & 0.0526 & 0.0512 & 0.0505 \\ [1ex] % [1ex] adds vertical space

\hline\hline \\ [-2ex] %inserts single line
\end{tabular}
\\
\small{Other parameters: $T_\delta=500$, $\alpha=0.05$, $\theta_0=0.1$, $\theta_1=0.2$, \\ $N=3$,  $\varrho=0.1$, $d=\beta=\sigma=1, \ \gamma=0$.}
%\\
%\small{Other parametric setting: $\theta_0=0.1$, $\theta_1=0.2$, $\sigma=1$, \\$d=1$, $\beta=1$, $\gamma=0$, $N=1$, $\rho=0.1$, $\Delta T=2000$.}
\label{table:VeriTh2.1T}
\end{table}

As already mentioned, the statistical test $R^\sharp_T$ derived in \cite{CialencoXu2013a}, while it is asymptotically the most powerful in $\cK^\sharp_\alpha$, it will not guarantee that the statistical errors will be below the threshold for a fixed finite $T$; only asymptotically it will be smaller than $\alpha$.
Indeed, as Table~\ref{table:VeriTh2.1T} shows, the Type~I error for $R^\sharp_T$ fluctuates around $\alpha=0.05$, with no pattern.
That was the very reason we proposed the tests $R^0$.

To illustrate the results from Theorem~\ref{Seconderr-Est}, and the behavior of Type~II error  $1-\bP_{\theta_1}^{N,T}\left(R_{T}^0\right)$, one needs to look at very large values of $T$, which is beyond our technical possibilities and the goal of this paper.
We will only give the results for some reasonable large values of $T$; see Table~\ref{table:VeriTh2.2T}. Note that indeed the Type~II error is decreasing as time $T$ gets larger. Also here, we show the corresponding results for the test $R^\sharp_T$.

\begin{table}[H]
\centering
\caption{Type~II errors for various $T$; Illustration of Theorem~\ref{Seconderr-Est}}
\begin{tabular}{c c c c c c c c c} % centered columns (5 columns)
\hline
$T$ & $10$ & $20$ & $30$ & $40$ & $50$ & $60$ \\ [0.5ex] % inserting body of the table
\hline \\ [-1ex]
$\exp\left(-\frac{(\theta_1-\theta_0)^2}{16\theta_0^2}MT\right)$ & $1.6\times10^{-4}$ & $2.5\times 10^{-8}$ & $4\times 10^{-12}$ & $6\times 10^{-16}$ & $10^{-19}$ & $1.6\times 10^{-23}$ \\ [1ex] % [1ex] adds vertical space
\hline \\ [-1ex]
$1-\bP_{\theta_1}^{N,T}\left(R_{T}^0\right)$ & 0.7155 & 0.3329 & 0.1148 & 0.0293 & 0.0070 & 0.0012 \\ [1ex]
\hline \\ [-1ex]
$1-\bP_{\theta_1}^{N,T}\left(R_{T}^\sharp\right)$ & 0.7946 &  0.2402 & 0.0457 & 0.0060 & 0.0006 & 0.0002 \\ [1ex]

\hline\hline \\ [-2ex] %inserts single line
\end{tabular}
\\
\small{Other parameters: $\alpha=0.05$, $\theta_0=0.1$, $\theta_1=0.2$, $N=3$,  $\varrho=0.1$, $d=\beta=\sigma=1, \ \gamma=0$.}
\label{table:VeriTh2.2T}
\end{table}

\subsection{Numerical tests for large number of Fourier modes}

Now we do a similar analysis by varying number of Fourier coefficients $N$, while the time horizon $T=1$ is fixed.
As mentioned above, the case of large $N$ is much more delicate, and as it turns out, according to the numerical results presented below, the error bounds for the statistical errors from Theorem~\ref{Firsterr-Est-N} are on conservative side.
The decay of the errors obtained in our numerical simulations is much faster than suggested by theoretical results, which from practical point of view is a desired feature.
%In Table~\ref{table:RhatN-I} and Table~\ref{table:RhatN-II} we present the results for the tests $\widehat{R}$ proposed in \cite{CialencoXu2013a}

\begin{table}[H]
\centering
\caption{Type~I errors for various $N$; Theorem \ref{Firsterr-Est-N}}
\begin{tabular}{c c c c c c c c c c c c} % centered columns (5 columns)
\hline \\ [-1ex] %inserts double horizontal lines
$N$ & $10$ & $20$ & $30$ & $40$ & $50$ & $60$ & $70$ & $80$\\ [0.5ex] % inserting body of the table
\hline \\ [-1ex]
$\bP_{\theta_0}^{N,T}\left(R_{N}^0\right)$ & 0.007 & 0.012 & 0.010 & 0.017 & 0.012 & 0.014 & 0.010 & 0.013 \\ [1ex] % [1ex] adds vertical space
\hline \\ [-2ex] %inserts single line
$\bP_{\theta_0}^{N,T}\left(R_{N}^\sharp\right)$ & 0.006 &  0.037 & 0.039 & 0.053 & 0.040 & 0.039 & 0.054 & 0.046 \\ [1ex] % [1ex] adds vertical space
\hline\hline \\ [-2ex] %inserts single line
\end{tabular}
\\
%\small{Other parameters: $\theta_0=0.1$, $\theta_1=0.2$, $\sigma=1$, $d=1$, $\beta=1$, $\gamma=0$, $\rho=0.1$.}
\small{Other parameters: $\alpha=0.05$, $\theta_0=0.1$, $\theta_1=0.2$, $T=1$,  $\varrho=0.1$, $d=\beta=\sigma=1, \ \gamma=0$.}
\label{table:VeriTh3.1}
\end{table}

\section{Concluding remarks}\label{sec:Conclusion}

\textsc{On discrete sampling.} Eventually, in real life experiments, the random field would be measured/sampled on a discrete grid, both in time and spatial domain.
It is true that the main results are based on continuous time sampling, and may appear as being mostly of theoretical interest.
However, as argued in the Section~\ref{sec:NumericalSim}, the main ideas of this paper and \cite{CialencoXu2013a} have a good prospect to be applied to the case of discrete sampling too. The error bounds of the numerical results presented herein contributes to the preliminary effort of studying the statistical inference problems for SPDEs in the discrete sampling  framework. At our best knowledge, there are no results on statistical inference for SPDEs with fully discretely observed data (both in time and space).
We outline here how to apply our results to discrete sampling, with strict proofs differed to our future studies.
If we assume that the first $N$ Fourier modes are observed at some discrete time points, then, to apply the theory presented here, one essentially has to approximate some integrals, including some stochastic integrals, convergence of each is well understood. Of course, the exact rates of convergence still need to be established.
The connection between discrete observation in space and the approximation of Fourier coefficients is more intricate. Natural way is to use discrete Fourier transform for such approximations. While intuitively clear that increasing the number of observed spacial points will yield to the computation of larger number of Fourier coefficients, it is less obvious, in our opinion, how to prove consistency of the estimators, asymptotic normality, and corresponding properties from hypothesis testing problem.

\bigskip
\noindent \textsc{On derivation of other tests.}
%As shown in Section~\ref{sec:largeT} and Section~\ref{sec:largeN}, using the large deviations type results (identity \eqref{LargeDev-Null} for $T$, and see also \cite[Section~3.2]{CialencoXu2013a} for large number of Fourier modes $N$), one can lead rejection regions for which there exist explicit formulas for controlling the upper bounds of statistical errors, in both regimes $T\to\infty$ and $N\to\infty$.
%The existence of explicit formulas comes at the cost that these tests are not the most powerful in their natural class.
We want to mention that the (sharp) large deviations, appropriately used,  can lead to other practically important family of tests.
%\textcolor[rgb]{0.00,0.07,1.00}{(This inspires that if we are willing to give up part of the convergence rate of Type~II error, the Type~I error will be small enough to control or may even converge rapidly to zero, with the large deviations appropriately used.)}
In fact, it is not difficult to observe that, if we take $R_T$ with
\begin{align*}
\eta\in\left(-\frac{(\theta_1-\theta_0)^2}{4\theta_0}M,\frac{(\theta_1-\theta_0)^2}{4\theta_1}M\right),
\end{align*}
then both Type I and Type II errors will go to zero, as $T\to\infty$.
Clearly, the motivation for doing this is to have both errors as small as possible.
Moreover, for such $\eta$ the statistical errors will go exponentially fast to zero.
Of course, this will not be the most powerful test in the sense of \cite{CialencoXu2013a}, since such chosen $\eta$ will reduce the exponential rate of convergence of Type II error. However, by shrinking the class of tests, one may preserve $R_T$ to be `asymptotically the most powerful' in the new class.
For example, once the asymptotical properties of errors are well understood, one can consider a new class of tests of the form
\begin{align*}
\cK_\alpha=\left\{(R_T): \limsup_{T\to\infty}\left(T^{\alpha_2} \exp\left(I(\eta)T+\eta T\right)\bP^{N,T}_{\theta_0}(R_T)-\alpha_0\right)T^{\alpha_3}\le\alpha_1\right\},
\end{align*}
where $\alpha_i$ ($0\le i\le 3$) are some parameters to be determined. Then, employing the same methodology as in \cite{CialencoXu2013a}, one can show that $R_T$ is the most powerful in $\cK_\alpha$, with only slight modification of some technical results. Similar ideas can lead to corresponding results for $N\to\infty$.

\bigskip
\noindent \textsc{On composite hypothesis.} Despite of the fact simple hypothesis testing problems are rarely used in practice, the efforts of this work, as well as those from  \cite{CialencoXu2013a}, should be seen as a starting point of a systematic study of general hypothesis testing problems and goodness of fit tests for stochastic evolution equation in infinite dimensional spaces.
As pointed out in \cite{CialencoXu2013a}, the developments of `asymptotic theory' for composite hypothesis testing problem will follow naturally, and consequently one can extend the results of this paper to the case of composite tests.

\section*{Acknowledgments}
We would like to thank the anonymous referees, the associate editor and the editor for their helpful comments and suggestions which improved greatly the final manuscript.
Igor Cialenco acknowledges support from the NSF grant DMS-1211256.

\newpage
{\small

\bibliographystyle{alpha}
%\bibliography{igor_bib_probability-10-02-2013}

\def\cprime{$'$}

}
\end{document}